\documentclass{article}
\usepackage{amsmath,amsthm,amssymb}

\newcommand{\PP}{\mathbb{P}}
\newcommand{\ZZ}{\mathbb{Z}}

\newcommand{\kk}{\mathbf{k}}

\newcommand{\OX}{\mathcal{O}_{X}}
\newcommand{\shO}{\mathcal{O}}
\newcommand{\shI}{\mathcal{I}}
\newcommand{\shF}{\mathcal{F}}
\newcommand{\shE}{\mathcal{E}}

\newcommand{\Pic}{\mathrm{Pic}}
\newcommand{\Num}{\mathrm{Num}}

\newcommand{\Supp}{\mathrm{Supp}\kern1.2pt}

\newcommand{\ch}{\mathrm{ch}\kern.8pt}
\newcommand{\td}{\mathrm{td}\kern.8pt}
\newcommand{\codim}{\mathrm{codim}\kern1.2pt}
\newcommand{\depth}{\mathrm{depth}\kern.8pt}

\newcommand{\charnum}{characteristic numbers $(d,\epsilon,\tau)$}

\newtheorem{lemma}{Lemma}[section]
\newtheorem{theorem}[lemma]{Theorem}
\newtheorem{proposition}[lemma]{Proposition}
\newtheorem{corollary}[lemma]{Corollary}

\theoremstyle{definition}

\newtheorem{remark}[lemma]{Remark}

\newcounter{cont}

\begin{document}

\title{Non-vanishing theorems for rank two vector \\ bundles on threefolds
\footnote{The paper was written while all authors were supported by MIUR (PRIN grant) and by local funds of their Universities, and were members of INdAM-GNSAGA.}}

\author{E.~Ballico \and P.~Valabrega \and M.~Valenzano} 

\date{}

\maketitle

\begin{abstract}
The paper investigates the non-vanishing of  $H^1(\shE(n))$, where $\shE$ is a (normalized) rank two vector bundle over any smooth irreducible threefold $X$ of degree $d$ such that $\Pic(X) \cong \ZZ$.   
If $\epsilon$ is the integer defined by the equality $\omega_X = \shO_X(\epsilon)$, and  $\alpha$ is the least integer $t$ such that $H^0(\shE(t)) \ne 0$, then, for a non-stable $\shE$ ($\alpha \le 0)$  the first cohomology module does not vanish at least between the endpoints $\frac{\epsilon-c_1}{2}$ and $-\alpha-c_1-1$. The paper  also  shows that  there are other non-vanishing intervals, whose endpoints depend on $\alpha$ and also on the second Chern class $c_2$ of $\shE$. 
If $\shE$ is stable the first cohomology module does not vanish at least between the endpoints $\frac{\epsilon-c_1}{2}$ and $\alpha-2$.
The paper considers also the case of a threefold $X$ with $\Pic(X) \ne \ZZ$ but $\Num(X) \cong \ZZ$ and gives similar non-vanishing results.
\\[8pt]
\\[8pt]
\end{abstract}

\section{Introduction}
In 1942  G. Gherardelli (\cite{Ghe}) proved that, if $C$ is a smooth irreducible curve in $\PP^3$ whose canonical divisors are cut out by the surfaces of some degree $e$ and moreover all linear series cut out by the surfaces in $\PP^3$ are complete, then $C$ is the complete intersection of two surfaces. Shortly and in the language of modern algebraic geometry: every $e$-subcanonical smooth curve $C$ in $\PP^3$ such that $h^1(\shI_C(n)) = 0$ for all $n$ is the complete intersection of two surfaces.

Thanks to the Serre correspondence between curves and vector bundles (see \cite{Hvb}, \cite{H1}, \cite{H2}) the above statement is equivalent to the following one:
if $\shE$ is a rank two vector bundle on $\PP^3$ such that $h^1(\shE(n)) = 0$ for all $n$, then $\shE$ splits.

There are many improvements  of the above result with a variety of different approaches (see for instance \cite{CV1}, \cite{CV2}, \cite{Ellia}, \cite{P}, \cite{RV}): it comes out that  a rank two vector bundle $\shE$ on $\PP^3$ is forced to split if $h^1(\shE(n))$ vanishes for just one strategic $n$, and such a value $n$ can be chosen arbitrarily within a suitable interval, whose endpoints depend on the Chern classes and the least number $\alpha$ such that $h^0(\shE(\alpha)) \ne 0$.

When rank two vector bundles on a smooth threefold $X$ of degree $d$ in $\PP^4$ are concerned, similar results can be obtained, with some interesting difference. 

In 1998 Madonna (\cite{Madonna}) proved that on a smooth threefold $X$ of degree $d$ in $\PP^4$ there are ACM rank two vector bundles (i.e. whose 1-cohomology vanishes for all twists) that do not split. And this can happen, for a normalized vector bundle $\shE$ ($c_1\in\{0,-1\})$, only when $ 1-\frac{d+c_1}{2} < \alpha < \frac{d-c_1}{2}$, while an ACM rank two vector bundle on $X$ whose $\alpha$ lies outside of the interval is forced to split. 

The following non-vanishing results for a normalized non-split rank two vector bundle on a smooth irreducible thereefold of degree $d$ in $\PP^4$ are proved in \cite{Madonna}:

if $\alpha \le 1-\frac{d+c_1}{2}$, then $h^1(\shE(\frac{d-3-c_1}{2}))\ne 0$ if $d+c_1$ is odd, $h^1(\shE(\frac{d-4-c_1}{2}))\ne 0, h^1(\shE(\frac{d-2-c_1}{2}))\ne 0$ if $d+c_1$ is even, while $h^1(\shE(\frac{d-c_1}{2}))\ne 0$ if $d+c_1$ is even and moreover $\alpha \le -\frac{d+c_1}{2}$;

if $\alpha\ge \frac{d-c_1}{2}$, then $h^1(\shE(\frac{d-3-c_1}{2}))\ne 0$ if $d+c_1$ is odd, while $h^1(\shE(\frac{d-4-c_1}{2}))\ne 0$ if $d+c_1$ is even.

In \cite{Madonna} it is also claimed that the same techniques  work to obtain similar non-vanishing results on any smooth threefold $X$ with $\Pic(X) \cong \ZZ$ and $h^1(\shO_X(n)) = 0$, for every $n$.

The present paper investigates the non-vanishing of  $H^1(\shE(n))$, where $\shE$ is a rank two vector bundle over any smooth irreducible threefold $X$ of degree $d$ such that $\Pic(X) \cong \ZZ$ and $H^1(\shO_X(n)) = 0, \forall n$. Actually we can prove that  for such an $\shE$  there is a wider range of non-vanishing for $h^1(\shE(n))$, so improving the above results.  

More precisely, when $\shE$ is (normalized and) non-stable ($\alpha \le 0$)  the first cohomology module does not vanish at least between the endpoints $\frac{\epsilon-c_1}{2}$ and $-\alpha-c_1-1$, where $\epsilon$ is defined by the equality $\omega(X) = \shO_X(\epsilon)$ (and is $d-5$ if $X \subset \PP^4$). But we can show that  there are other non-vanishing intervals, whose endpoints depend on $\alpha$ and also on the second Chern class $c_2$ of $\shE$. 

If on the contrary $\shE$ is stable  the first cohomology module does not vanish at least between the endpoints $\frac{\epsilon-c_1}{2}$ and $\alpha-2$, but other ranges of non-vanishing can be produced.

We give a few examples obtained by pull-back from vector bundles on $\PP^3$.

We must remark that most of our non-vanishing results do not exclude the range for $\alpha$ between the endpoints $1-\frac{d+c_1}{2}$ and $\frac{d-c_1}{2}$ (for a general threefold it becomes $-\frac{\epsilon+3+c_1}{2} < \alpha < \frac{\epsilon+5-c_1}{2})$.  Actually  \cite{Madonna} produces some examples of nonsplit ACM rank two vector bundles on smooth hypersurfaces in $\PP^4$, but it can be seen that they do not conflict with our theorems.

As to threefolds with $\Pic(X) \ne \ZZ$, we need to observe that a key point is a good definition of the integer $\alpha$. We are able to prove, by using a boundedness argument, that $\alpha$ exists when $\Pic(X) \ne \ZZ$ but $\Num(X) \cong \ZZ$. In this event the correspondence between rank two vector bundles and two-codimensional subschemes can be proved to hold.  In order to obtain non-vanishing results that are similar to the results proved when $\Pic(X) \cong \ZZ$, we need also use the Kodaira vanishing theorem, which holds in characteristic 0.   We can extend the results to characteristic $p > 0$ if we assume a Kodaira-type vanishing condition.
\section {Notation}

We work over an algebraically closed field $\kk$ of any characteristic.  
\\
Let $X$ be a non-singular irreducible projective algebraic variety of dimension 3, for short a smooth  threefold.
\\
We fix an ample divisor $H$ on $X$, so we consider the polarized threefold $(X,H)$. 
\\
We denote with $\OX(n)$, instead of $\OX(nH)$, the invertible sheaf corresponding to the divisor $nH$, for each $n\in\ZZ$.
\\
For every cycle $Z$ on $X$ of codimension $i$ it is defined its degree with respect to $H$, i.e.
$\deg(Z;H) := Z \cdot H^{3-i}$,
having identified a codimension 3 cycle on $X$, i.e. a $0$-dimensional cycle, with its degree, which is an integer. \\
From now on (with the exception of section 7) we consider a smooth polarized threefold $(X,\shO_{X}(1))=(X,H)$ that 
satifies the following conditions:
\begin{description}
\item[\qquad$\mathbf{(C1)}$] $\Pic(X)\cong\ZZ$ generated by $[H]$,
\item[\qquad$\mathbf{(C2)}$] $H^1(X,\OX(n)) = 0$ for every $n\in\ZZ$,
\item[\qquad$\mathbf{(C3)}$] $h^0(O_X(1)) \ne 0$.
\end{description}
By condition $\mathbf{(C1)}$ every divisor on $X$ is linearly equivalent to $aH$ 
for some integer $a\in\ZZ$, i.e.\ every invertible sheaf on $X$ is (up to an isomorphism) of type $\OX(a)$ for some  $a\in\ZZ$, in particular we have for the canonical divisor
$K_X \sim \epsilon H$, or equivalently $\omega_{X}\simeq\shO_{X}(\epsilon)$, 
for a suitable integer $\epsilon$.
Furthermore, by Serre duality condition $\mathbf{(C2)}$ implies that $H^2(X,\OX(n)) = 0$ for all $n\in\ZZ$.
\\
Since by assumption $A^{1}(X)=\Pic(X)$ is isomorphic to $\ZZ$ through the 
map $[H]\mapsto 1$, where $[H]=c_{1}(\shO_{X}(1))$, we identify the first 
Chern class  $c_1(\shF)$ of a coherent sheaf with a whole number $c_1$, where $c_1(\shF) = c_1 H$.
\\
The second Chern class $c_2(\shF)$ gives the integer $c_2 = c_2(\shF)\cdot H$ and we will call this integer the second Chern number or the second Chern class of $\shF$.
\\
We set
$$d := \deg(X;H) = H^3,$$
so $d$ is the \lq\lq degree\rq\rq\ of the threefold $X$ with respect to the ample divisor $H$.
\\
Let $c_1(X)$ and $c_2(X)$ be the first and second Chern classes of $X$, that is of its tangent bundle $TX$ (which is a locally free sheaf of rank 3); then we have
$$c_1(X) = [-K_X] = -\epsilon [H],$$
so we identify the first Chern class of $X$ with the integer $-\epsilon$. Moreover we set
$$\tau := \deg(c_2(X);H) = c_2(X) \cdot H,$$
i.e.\ $\tau$ is the degree of the second Chern class of the threefold $X$.
\\
In the following we will call the triple of integers $(d,\epsilon,\tau)$ the \textbf{characteristic numbers} of the polarized threefold $(X,\OX(1))$.

We recall the well-known Riemann-Roch formula on the threefold $X$ (see \cite{Valenzano}, proposition 4).
\begin{theorem}[Riemann-Roch]\label{gRR}
Let $\shF$ be a rank $r$ coherent sheaf on $X$ with Chern classes 
$c_{1}(\shF)$, $c_{2}(\shF)$ and $c_{3}(\shF)$. Then the 
Euler-Poincar\'e characteristic of $\shF$ is
\begin{align*}
\chi(\shF) = & \frac{1}{6}\Big(c_{1}(\shF)^{3} - 3 c_{1}(\shF)\cdot 
c_{2}(\shF) + 3 c_{3}(\shF)\Big) + 
\frac{1}{4}\Big(c_{1}(\shF)^{2} - 2 c_{2}(\shF)\Big)\cdot 
c_{1}(X) + \\
& + \frac{1}{12} c_{1}(\shF)\cdot\Big(c_{1}(X)^{2} + c_{2}(X)\Big) +
\frac{r}{24} c_{1}(X)\cdot c_{2}(X)
\end{align*}
where $c_{1}(X)$ and $c_{2}(X)$ are the Chern classes of $X$, that is 
the Chern classes of the tangent bundle $TX$ of $X$.
\end{theorem}

So applying the Riemann-Roch Theorem to the invertible sheaf $\OX(n)$, for each $n\in\ZZ$, we get the Hilbert polynomial of the sheaf $\OX(1)$
\begin{equation}
\chi(\OX(n)) = \frac{d}{6} \left(n - \frac{\epsilon}{2}\right) \left[ \left(n - \frac{\epsilon}{2}\right)^2 + \frac{\tau}{2d} - \frac{\epsilon^2}{4} \right]\!.
\end{equation}

Let $\shE$ be a rank 2 vector bundle on the threefold $X$ with Chern classes $c_1(\shE)$ and $c_2(\shE)$, so with Chern numbers  $c_1$ and $c_2$. 
We assume that $\shE$ is normalized, i.e.\ that $c_1 \in\{0,-1\}$.
It is defined the integer  $\alpha$, the so called first relevant level, such that $h^0(\shE(\alpha)) \ne 0, h^0(\shE(\alpha-1)) = 0$. If $\alpha > 0$, $\shE$ is called stable, non-stable otherwise.

We set
$$\vartheta = \frac{3c_2}{d} - \frac{\tau}{2d} + \frac{\epsilon^2}{4}-\frac{3c_1^2}{4},  \qquad
\zeta_0 = \frac{\epsilon-c_1}{2}, \quad \text{and} \quad w_0 = [\zeta_0]+1,$$
where $[\zeta_0]=$ integer part of $\zeta_0$, 
so the Hilbert polynomial of $\shE$ can be written as
$$\chi(\shE(n)) = \frac{d}{3}\big(n - \zeta_0\big)\Big[\big(n - \zeta_0\big)^2 - \vartheta\Big].$$
If $\vartheta\ge0$ we set
$$\zeta = \zeta_0 + \sqrt{\vartheta}$$
so in this case the Hilbert polynomial of $\shE$ has the three real roots
$\zeta' \le \zeta_0 \le \zeta$
where $\zeta' =  \zeta_0 - \sqrt{\vartheta}, \zeta =  \zeta_0 + \sqrt{\vartheta}$.
We also define $\bar\alpha = [\zeta]+1$. \\
The polinomial $\chi(\shE(n))$, as a rational polynomial,
 has three real roots if and only if $\vartheta\ge0$, and it has only one real root if and only if $\vartheta<0$.
\\
If $\shE$ is normalized, we set $$\delta = c_2+d\alpha^2+c_1d\alpha.$$ 
 \begin{remark}We have $\delta = 0$ if and only if $\shE$ splits (see \cite{VV}, Lemma 3.13: the proof works in general).\end{remark}

Unless stated otherwise, we work over the smooth polarized threefold $X$ and \emph{$\shE$ is a normalized non-split rank two vector bundle on $X$}.

\section{About the characteristic numbers $\epsilon$ and $\tau$}
  
 In this section we want to recall some essentially known properties of the characteristic numbers of the threefold $X$  (see also \cite{SB} for more general statements).   
 We start with the following remark.

\begin{remark}\label{OX}
1.\ For the fixed ample invertible sheaf $\OX(1)$ we have
$$h^0(\OX(n)) \begin{cases} = 0 & \quad\text{for } n < 0 \\
= 1 & \quad \text{for } n = 0 \\
\ne 0 & \quad\text{for } n > 0 \end{cases}$$
and also $h^0(\OX(m)) - h^0(\OX(n)) > 0$ for all $n,m\in\ZZ$ with $m > n \ge 0$. 
\\
2.\ It holds
$$\chi(\OX) = h^0(\OX) - h^3(\OX) = 1 - h^0(\OX(\epsilon)),$$
so we have:
$$\chi(\OX) = 1 \iff \epsilon < 0, \ \ \
\chi(\OX) = 0 \iff \epsilon = 0, \ \ \
\chi(\OX) < 0 \iff \epsilon > 0.$$
\end{remark}

\begin{proposition}
Let $(X,\OX(1))$ be a smooth polarized threefold with \charnum. Then it holds:
\begin{description}
\item[1)] $\epsilon \ge -4$,
\item[2)] $\epsilon = -4$ if and only if $X = \PP^3$, i.e. $(d,\epsilon,\tau)=(1,-4,6)$ and so $\frac{\tau}{2d} - \frac{\epsilon^2}{4} = -1$,
\item[3)] if $\epsilon = -3$, then $X$ is a hyperquadric in $\PP^4$, so $(d,\epsilon,\tau)=(2,-3,8)$ and   $\frac{\tau}{2d} - \frac{\epsilon^2}{4} = -\frac{1}{4}$, 
\item[4)] $\epsilon\tau$ is a multiple of $24$, in particular if $\epsilon < 0$ then $\epsilon \tau = -24$,
\item[5)] if $\epsilon\ne 0$, then $\tau > 0$,
\item[6)] if $\epsilon = 0$, then $\tau > -2d$,
\item[7)] $\tau$ is always even,
\item[8)] if $\epsilon$ is even, then\ \
$\frac{\tau}{2d} - \frac{\epsilon^2}{4} \ge -1$,
\item[9)] if $\epsilon$ is odd, then\ \
$\frac{\tau}{2d} - \frac{\epsilon^2}{4} \ge -\frac{1}{4}$,
\item[10)] if $\epsilon < 0$, then the only possibilities for $(\epsilon,\tau)$ are the following
$$(\epsilon,\tau) \in \{(-4,6), \, (-3,8), \, (-2,12), \, (-1,24) \},$$

\end{description}
\end{proposition}
\begin{proof} 

For statements \textbf{1)}, \textbf{2)}, \textbf{3)} see \cite{SB}.
\\
\textbf{4)} Observe that $\chi(\OX) = - \frac{1}{24} \epsilon \tau$ is an integer, and moreover, if $\epsilon < 0$, then $\chi(\OX)=1$.
\\
\textbf{5)} By Remark~\ref{OX} we have: if $\epsilon > 0$ then $- \frac{1}{24} \epsilon \tau < 0$, while if $\epsilon < 0$ then $- \frac{1}{24} \epsilon \tau > 0$. In both cases we deduce $\tau > 0$.
\\
\textbf{6)} If $\epsilon = 0$, then we have
$$\chi(\OX(n)) = \frac{d}{6} n  \left(n^2 + \frac{\tau}{2d} \right),$$
and also
$$\chi(\OX(n)) = h^0(\OX(n)) > 0 \quad \forall n > 0,$$
therefore we must have $\frac{2d+\tau}{12} > 0$, so $\tau > -2d$.
\\
\textbf{7)} Assume that $\epsilon$ is even, then we have
$$d\left(1- \frac{\epsilon}{2}\right)\left(1+ \frac{\epsilon}{2}\right) +  \frac{\tau}{2} = 
d\left(1-\frac{\epsilon^2}{4}+\frac{\tau}{2d}\right) = 6\, \chi\left(\OX\left(\frac{\epsilon}{2}+1\right)\right) \in\ZZ$$
and moreover $d\left(1- \frac{\epsilon}{2}\right)\left(1+ \frac{\epsilon}{2}\right)\in\ZZ$, so $\tau$ must be even. \\
If $\epsilon$ is odd, the proof is quite similar.
\\
\textbf{8)} Let $\epsilon$ be even. If it holds
$$h^0\left(\OX\!\left(\frac{\epsilon}{2}+1\right)\right) - h^0\left(\OX\!\left(\frac{\epsilon}{2}-1\right)\right) = \chi\left(\OX\left(\frac{\epsilon}{2}+1\right)\right) < 0,$$
then we must have $h^0\left(\OX\left(\frac{\epsilon}{2}-1 \right)\right) \ne 0$, which implies 
$$h^0\left(\OX\!\left(\frac{\epsilon}{2}+1\right)\right) - h^0\left(\OX\!\left(\frac{\epsilon}{2}-1\right)\right) \ge 0,$$ a contradiction. 
So we must have:$$  \chi\left(\OX\!\left(\frac{\epsilon}{2}+1\right)\right) =
 \frac{d}{6}\left(1 + \frac{\tau}{2d} - \frac{\epsilon^2}{4}\right) \ge 0,$$
therefore
$$\frac{\tau}{2d} - \frac{\epsilon^2}{4} \ge -1.$$
\\
\textbf{9)}  The proof is quite similar to the proof of \textbf{8}).
\\
\textbf{10)} If $\epsilon < 0$, then by \textbf{1)} we have $\epsilon\in\{-4,-3,-2,-1\}$, and moreover $\epsilon \tau = -24$ by \textbf{4)}, so we obtain the thesis.
\end{proof}

\section{Non-stable vector bundles ($\alpha \le 0$)}

\noindent
\textbf{Case $\epsilon \ge 1$.}\\
In this subsection we make the following assumptions:

\emph{$\shE$ is a normalized non-split  rank two vector bundle with $\alpha \le 0$ and $\epsilon \ge 1$}. 
\\
The case $\epsilon \le 0$ is investigated in the next subsection.
\medskip
  
\begin{proposition}\label{nonstable0}
Assume that $\zeta_0 < -\alpha-c_1-1$. Then it holds: $$h^1(\shE(n))-h^2(\shE(n)) = (n-\zeta_0)\delta$$  for every $n$ such that $\zeta_0 < n \le -\alpha-c_1-1$.
\end{proposition}

\begin{proof} 
First we assume $c_1 = 0$. 
It is enough to observe that, from the inequality $n+\alpha \le 0$ and the exact sequence 
$$0 \to \shO_X(n-\alpha) \to \shE(n) \to \shI(n+\alpha) \to 0$$
we obtain: $h^0(\shE(n)) = h^0(\shO_X(n-\alpha)) = \chi(\shO_X(n-\alpha))+h^0(\shO_X(\epsilon-n+\alpha)) =  \chi(\shO_X(n-\alpha))$ since $\epsilon-n+\alpha \le -1$.  We also have: $h^0(\shE(\epsilon-n)) = h^0(\shO_X(\epsilon-n-\alpha)) = \chi(\shO_X(\epsilon-n-\alpha))+h^0(\shO_X(n+\alpha)) =  \chi(\shO_X(\epsilon-n-\alpha))$.
\\
Now it is enough to observe that  $h^1(\shE(n))-h^2(\shE(n)) = h^0(\shE(n))-h^0(\shE(\epsilon-n)) -\chi(\shE(n)) =  \chi(\shO_X(n-\alpha))- \chi(\shO_X(\epsilon-n-\alpha))-\chi(\shE(n))$. If we use the Riemann-Roch formulas for the Euler functions we obtain the required equality.
\\
Now we assume $c_1 = -1$.
We recall that $h^3(\shE(n)) = h^0(\shE(\epsilon-n+1))$. As before we have $h^1(\shE(n))-h^2(\shE(n)) =  \chi(\shO_X(n-\alpha))- \chi(\shO_X(\epsilon-n-\alpha+1))-\chi(\shE(n))$ and the computation is very similar.
\end{proof}

\begin{remark} Observe that the statement of Proposition \ref{nonstable0} still holds  when $n = \zeta_0$, because the two sides of the equality vanish.  
\end{remark}

\begin{theorem}\label{nonstable1} 
Let us assume that $\zeta_0 < -\alpha-c_1-1$ and let $n$ be such that $\zeta_0 < n \le -\alpha-1-c_1$. Then $h^1(\shE(n)) \ge (n-\zeta_0)\delta$. In particular  $h^1(\shE(n)) \ne 0$.
\end{theorem}

\begin{proof}
It is enough to observe that $$h^1(\shE(n))-h^2(\shE(n)) = (n-\zeta_0)\delta$$ and that the right side of the equality is strictly positive for a non-split vector bundle.
\end{proof}

\begin{remark}
Observe that the above theorem describes  a non-empty set of integers if and only  if $-\alpha -1 -c_ 1 > \zeta_0$; this means  $\alpha < -\frac{\epsilon+2+c_1}{2}$, i.e.\ $\alpha \le -\frac{\epsilon+3+c_1}{2}$. So our assumption on $\alpha$ agrees with the bound of \cite{Madonna}.
\\
Observe that the inequality on $\alpha$ implies that $\alpha \le -2$ if $\epsilon \ge 1$.

 \end{remark}

\begin{theorem} \label{nonstable2}
Assume that $\frac{6\delta}{d}-\frac{\tau}{2d}+\frac{\epsilon^2}{4}-\frac{3c_{1}^2}{4}\ge 0$. Let $n > \zeta_0$ be such that $ \epsilon-\alpha-c_1+1 \le n < \zeta_0+\sqrt{\frac{6\delta}{d}-\frac{\tau}{2d}+\frac{\epsilon^2}{4}-\frac{3c_{1}^2}{4}}$ and put 
$$S(n) =\frac{d}{6}\left(n-\frac{\epsilon-c_1}{2}\right)\!\left[(n-\frac{\epsilon-c_1}{2})^2-6\frac{c_2+d \alpha^2+c_1d\alpha}{d}+\frac{\tau}{2d}-\frac{\epsilon^2}{4}+\frac{3c_1^2}{4}\right]\!.$$
Then $h^1(\shE(n))\ge -S(n) > 0$. In particular $h^1(\shE(n))\ne 0$.
\end{theorem}
\begin{proof}
Assume $c_1 = 0$. Under our hypothesis $h^0(\shE(\epsilon-n)) = 0$ and so $-h^1(\shE(n))+h^2(\shE(n)) = \chi(\shE(n)) -  h^0(\shO_X(n-\alpha))$. Observe that $\chi(\shE(n)) - h^0(\shO_X(n-\alpha)) -S(n) = \frac{1}{2}nd\alpha(-\epsilon+n+\alpha))+\frac{1}{12}d\alpha(-3\epsilon \alpha+2\alpha^2+\epsilon^2+\frac{\tau}{d})  \le 0$. Therefore we have:  $h^1(\shE(n)) \ge h^2(\shE(n))-S(n)$. Hence $h^1(\shE(n))$ may possibly  vanish  when 
$$\left(n-\frac{\epsilon}{2}\right)^2-6\frac{c_2+d \alpha^2}{d}+\frac{\tau}{2d}-\frac{\epsilon^2}{4} \ge 0.$$ When $S(n) < 0$, so $-S(n) > 0$, $h^1(\shE(n)) \ge -S(n) > 0$ and in particular it cannot vanish.
\\
If $c_1 = -1$ the proof is quite similar.
\end{proof}

Now we put $\frac{\tau}{2d}-\frac{\epsilon^2}{4} =\lambda$ and consider the following degree $3$ polynomial: 
$$F(X) = X^3+\left(\lambda-\frac{6\delta}{d}\right)X+\frac{6\alpha \delta}{d}.$$
It is easy to see that, if $\frac{6\delta}{d}-\frac{\tau}{2d}+\frac{\epsilon^2}{4} \le 0$,  $F(X)$ is strictly increasing and so it has only one real root $X_0$.

\begin{theorem} \label{nonstable3}
Assume that $\frac{6\delta}{d}-\frac{\tau}{2d}+\frac{\epsilon^2}{4} \le 0$. Let $n$ be such that $ \epsilon-\alpha-c_1+1 \le n < -\alpha+X_0+\zeta_0$, where $X_0$ = unique real root of $F(X)$ . Then $h^1(\shE(n)) \ge -\frac{d}{6}F(n+\alpha-\zeta_0) > -\frac{d}{6}F(X_0) = 0$. In particular  $h^1(\shE(n)) \ne 0$. 
\end{theorem}
\begin{proof}
Assume $c_1 = 0$, the proof being quite similar if $c_1 = -1$.
\\
It holds (see proposition \ref{nonstable0}): 
\begin{align*}
h^1(\shE(n)) & - h^2(\shE(n)) = \chi(\shO_X(n-\alpha))-\chi(\shE(n)) = \\
& = \left(n-\frac{\epsilon}{2}\right)(c_2+d\alpha^2)+\chi(\shO_X(\epsilon-n-\alpha)) = \\
& = \left(n-\frac{\epsilon}{2}\right)(c_2+d\alpha^2)-\frac{d}{6}\left(n+\alpha-\frac{\epsilon}{2}\right)\left(\left(n+\alpha-\frac{\epsilon}{2}\right)^2+\lambda\right).
\end{align*}
If we put: $X = n+\alpha-\frac{\epsilon}{2}$, we obtain: $\frac{d}{6} F(X) =  \frac{d}{6}(X^3+(\lambda-\frac{6\delta}{d})X+\frac{6\alpha \delta}{d}) = h^2(\shE(n))-h^1(\shE(n))$. Therefore $h^1(\shE(n)) > -\frac{d}{6}F(n+\alpha-\zeta_0) > -\frac{d}{6}F(X_0) = 0$.  

\end{proof}

\begin{remark}
Observe that in Theorems \ref{nonstable2} and \ref{nonstable3} $\alpha$ can be $0$.  
\end{remark}

\medskip

\noindent
\textbf{Case $\epsilon \le 0$.}

\begin{remark}In the event that $\epsilon \le -2$, we have $\epsilon-\alpha-c_1+1 \le -\alpha-c_1-1$.   Therefore Theorems \ref{nonstable2}, \ref{nonstable3} give something new only beyond $-\alpha-c_1-1.$ 
\end{remark}

First of all we observe that Theorems \ref{nonstable1}, \ref{nonstable3}  obviously hold as they are stated also when  $\epsilon \le 0$.  So we discuss Theorem \ref{nonstable2}.
\\
A. $\epsilon \le -2$.
\\
In theorem \ref{nonstable2} we need to know that
$$\frac{1}{2}nd\alpha(-\epsilon+n+\alpha)+\frac{1}{12}d\alpha(\epsilon^2+\frac{\tau}{d}-3\epsilon \alpha+2\alpha^2)  \le 0.$$
The first term of the sum is for sure negative; as for
$$\frac{1}{12}d\alpha\left(\epsilon^2+\frac{\tau}{d}\right)+\frac{1}{12}d\alpha^2(-3\epsilon+2\alpha)$$
we observe that the quantity in brackets has discriminant
$$\Delta = \epsilon^2-8\frac{\tau}{d} = 4\left(\frac{\epsilon^2}{4}-\frac{\tau}{2d}+\frac{\tau}{2d}-8\frac{\tau}{d}\right) \le 4(1-15) < 0.$$
Therefore it is positive for all $\alpha \le 0$ and the product is negative.
\\
B. $\epsilon = -1$.
\\
In theorem \ref{nonstable2} we need to know that
$$\frac{1}{2}nd\alpha(1+n+\alpha)+\frac{1}{12}d\alpha\left(1+\frac{\tau}{d}\right)+\frac{1}{12}d\alpha^2(3+2\alpha)  \le 0.$$  
If $\alpha \le -2$, then it is enough to observe that $\frac{\tau}{d}+3\alpha+2\alpha^2 \ge 0$. If $\alpha = -1$ we have to consider $-\frac{1}{2}n^2d+\frac{1}{12}d\frac{\tau}{d}$ and then we observe that $6n^2+\frac{\tau}{d} > 0$. If $\alpha = 0$ obviously the quantity is $0$.
\\
C. $\epsilon = 0$.
\\
In theorem \ref{nonstable2} we need to know that $$ \frac{1}{2}nd\alpha(n+\alpha)+\frac{1}{12}d\alpha\left(\frac{\tau}{d}\right)+\frac{1}{12}d\alpha^2(2\alpha)  \le 0.$$  
\\
It is enough to observe that $2\alpha^2+\frac{\tau}{d} > 0$ by Proposition 3.2, $\mathbf{6)}$ if $\alpha < 0$; otherwise we have a $0$ quantity, and that $n+\alpha \le 0$.

\begin{remark}
Observe that the case $\alpha = 0$ in Theorem \ref{nonstable1} can occur only if  $\epsilon \le -c_1-3$.
\end{remark}

\begin{remark}
In theorem \ref{nonstable2} we do not use the hypothesis $-\frac{\epsilon+3}{2} \ge \alpha$, but we assume that $6\frac{c_2+d \alpha^2}{d}-\frac{\tau}{2d}+\frac{\epsilon^2}{4}-1\ge 0$. 
In theorem \ref{nonstable3} we do not use the hypothesis $-\frac{\epsilon+3}{2} \ge \alpha$, but we assume that $6\frac{c_2+d \alpha^2}{d}-\frac{\tau}{2d}+\frac{\epsilon^2}{4} < 0$.  Moreover in both theorems there is a  range for $n$, the left endpoint being  $\epsilon-\alpha-c_1+1$ and the right endpoint being either $\zeta_0+\sqrt{6\frac{c_2+d \alpha^2}{d}-\frac{\tau}{2d}+\frac{\epsilon^2}{4}-1}$ (\ref{nonstable2}) or $\zeta_0-\alpha+X_0$ (\ref{nonstable3}).
\\
In \cite{Madonna} there are examples of ACM nonsplit vector bundles on smooth threefolds in $\PP^4$, with $-\frac{\epsilon+3+c_1}{2} < \alpha < \frac{\epsilon+5-c_1}{2}$. 
We want to emphasize that our theorems do not conflict with the examples of \cite{Madonna}: if $C$ is any curve described in \cite{Madonna} and lying on a smooth threefold of degree $d$, then our numerical constraints cannot be satisfied (we have checked it directly in many but not all cases).
\end{remark}

\begin{remark}
Let us consider a smooth degree $d$ threefold $X \subset \PP^4$. 
\\  
We  have: 
$$\epsilon = d-5,\ \ \tau =  d(10-5d+d^2),\ \  \theta = \frac{3c_2}{d}-\frac{d^2-5+3c_1^2}{4}$$
(see \cite{Valenzano}).
As to the characteristic function of $O_X$ and $\shE$, it holds:
$$\chi(\shO_X(n)) = \frac{d}{6}\left(n-\frac{d-5}{2}\right)\!\left[\left(n-\frac{d-5}{2}\right)^2+\frac{d^2-5}{4}\right]\!,$$
$$\chi(\shE(n)) = \frac{d}{3} \left(n-\frac{d-5-c_1}{2}\right)\!\left[\left(n-\frac{d-5-c_1}{2}\right)^2+\frac{d^2}{4}-\frac{5}{4}+\frac{3c_1^2}{4}-\frac{3c_2}{d}\right]\!.$$
Then it is easy to see that the hypothesis of Theorem \ref{nonstable2}, i.e. $6\frac{\delta}{d}-\frac{d^2-5+3c_1^2}{4}\ge 0$ is for sure fulfilled if $c_2 \ge 0, \alpha \le -\frac{d-2+c_1}{2}$. In fact we  have (for the sake of simplicity when $c_1 = 0)$: $-6\frac{6c_2+d\alpha^2}{d}+\frac{d^2-5}{4} \le \frac{d^2-5}{4}-6\frac{d^2-2d+1}{4} = -\frac{5d^2-12d+11}{4} < 0$.
\end{remark}
\begin{remark}  
Condition $\mathbf{(C2)}$ holds for sure if $X$ is a smooth hypersurface of $ \PP^4$. In general, for a characteristic $0$ base field, only the Kodaira vanishing holds (\cite{HAL}, remark 7.15) and so, unless we work over a threefold $X$ having  some stronger vanishing, we need assume, in Theorems \ref{nonstable1}, \ref{nonstable2}, \ref{nonstable3} that $n-\alpha \notin \{0,...,\epsilon\}$ (which implies, by duality, that also $\epsilon -n+\alpha \notin \{0,...,\epsilon\}$).

Observe that the first assumption ($n-\alpha \notin \{0,...,\epsilon\})$ in the case of Theorem \ref{nonstable1} is automatically fulfilled because of the hypothesis $\zeta_0 < -\alpha-c_1-1$, and in Theorems \ref{nonstable2} and \ref{nonstable3} because of  the hypothesis $\epsilon-\alpha-c_1+1 \le n$. In fact $n-\alpha$ is greater than $\epsilon$. But this implies that $\epsilon-n+\alpha < 0$ and so also the second condition is fulfilled, at least when $\epsilon \ge 0$. For the case $\epsilon < 0$ in positive characteristic see \cite{SB}.

Observe that, if $\epsilon < 0$, Kodaira (and so $\mathbf{(C2)}$) holds for every $n$.

 For a general discussion, also in characteristic $p > 0$, of this question, see section 7, remark 7.8.
\end{remark}
\begin{remark} In the above theorems we assume that $\shE$ is a nonsplit bundle. If $\shE$ splits, then (see section 2) $\delta = 0$. In Theorem \ref{nonstable1} this implies $h^1(\shE(n))-h^2(\shE(n)) = 0$ and so nothing can be said on the non-vanishing.

Let us now consider Theorem \ref{nonstable2}. If $\delta = 0$, then we must have: $\zeta_0 < n < \zeta_0+\sqrt{-\frac{\tau}{2d}+\frac{\epsilon^2}{4}-\frac{3c_1^2}{4}} \le \zeta_0+1$ (the last inequality depending on Proposition 3.2, $\mathbf{8),9)})$.  As a consequence $\zeta_0$ cannot be a whole number. Moreover, since we have $2\zeta_0-\alpha+1 \le n < \zeta_0+\sqrt{-\frac{\tau}{2d}+\frac{\epsilon^2}{4}-\frac{3c_1^2}{4}}$, we obtain that $\zeta_0 < \alpha \le 0$, hence $\epsilon-c_1 \le -1$. If $c_1 = 0$, $\epsilon \in \{-1, -3\}$. If $\epsilon = -3$,  then $n$ must satisfy (see Proposition 3.2, $\mathbf{8})$ the following inequalities: $-\frac{3}{2} < n < -1$, which is a contradiction. If $\epsilon = -1$, then, by Proposition 3.2, $\mathbf{8})$ we have $-1+\alpha+1 < -\frac{1}{2}+ \frac{1}{2} = 0$, which implies $\alpha > 0$, a contradiction.  If $c_1 = -1$, then $\epsilon \in \{-2, -4\}$. If $\epsilon = -4$, we have $\sqrt{-\frac{\tau}{2d}+\frac{\epsilon^2}{4}-\frac{3c_1^2}{4}} = \frac{1}{2}$, and so   we must have: $-\frac{3}{2} < n < -1$, which is impossible. If $\epsilon = -2$, then $\zeta_0 = -\frac{1}{2}$ and so $-2-\alpha+2 < -\frac{1}{2}+\sqrt{1-\frac{3}{4}}$, which implies $-\alpha < 0$ hence $\alpha > 0$, a contradiction with the non-stability of $\shE$.

Then we consider Theorem \ref{nonstable3}. The vanishing of $\delta$ on the one hand implies  $\lambda > 0$ and $X_0 = 0$.  But on the other hand from our hypothesis on the range of $n$ we see that $\zeta_0 \le -2$, hence $\epsilon = -4, c_1 = 0$. But this contradicts Proposition 3.2, $\mathbf{2)}$.
\end{remark}

\section{Stable vector bundles} 
In the present section we assume that $\alpha \ge \frac{\epsilon-c_1+5}{2}$, or equivalently that $c_1+2\alpha\ge\epsilon+5$. This means that $\alpha \ge 1$ in any event, so $\shE$ is stable.

The following lemma holds both in the stable and in the non-stable case.

\begin{lemma} \label{leftvanishing}
 If $h^1(\shE(m)) = 0$ for some integer $m \le \alpha-2$, then $h^1(\shE(n)) = 0$ for all $n \le m$.
\end{lemma}
\begin{proof}
First of all observe that, by our condition $\mathbf{(C3)}$, from the restriction exact sequence we can obtain in cohomology the exact sequence 
$$0 \to H^0(\shE(t-1) \to H^0(\shE(t)) \to H^0(\shE_H(t)) \to 0.$$ Then we can follow the proof given in \cite{VV} for $\PP^3$ (where condition $\mathbf{(C3)}$  is automatically fulfilled).
\end{proof}

\begin{theorem}\label{stable1}
Let $\shE$ be a rank 2 vector bundle on the threefold $X$ with first relevant level $\alpha$.
If $\alpha\ge\frac{\epsilon+5-c_1}{2}$, then $h^1(\shE(n))\ne 0$ for $w_0\le n\le \alpha-2$.
\end{theorem}
\begin{proof}
By the hypothesis it holds $w_0 \le \alpha-2$, so we have $h^0(\shE(n))=0$ for all $n\le w_0+1$.
Assume $h^1(\shE(w_0))=0$, then by Lemma~\ref{leftvanishing} it holds $h^1(\shE(n))=0$ for every $n\le w_0$. Therefore we have
$$\chi(\shE(w_0)) = h^0(\shE(w_0)) + h^1(\shE(-w_0+\epsilon-c_1)) - h^0(\shE(-w_0+\epsilon-c_1)) = 0.$$ Now observe that the characteristic function has at most three real roots, that are symmetric with respect to $\zeta_0$. Therefore, if $w_0$ is a root, then $w_0 = \zeta_0+\sqrt{\theta}$ and the other roots are $\zeta_0$ and  $ \zeta_0-\sqrt{\theta}$. This implies that $\chi(\shE(w_0+1)) > 0$. On the other hand
$$\chi(\shE(w_0+1)) =  - h^1(\shE(w_0+1)) \le 0,$$
a contradiction. So we must have $h^1(\shE(w_0))\ne0$, then by Lemma \ref{leftvanishing} we obtain the thesis.
\end{proof}

\begin{corollary}
If $\shE$ is \emph{ACM}, then $\alpha<\frac{\epsilon+5-c_1}{2}$.
\end{corollary}

\begin{theorem} \label{stable2}
Let $\shE$ be a normalized rank 2 vector bundle on the threefold $X$ with $\vartheta\ge0$, then the following hold:
\begin{description}
\item[1)] $h^1(\shE(n))\ne 0$ for $\zeta_0< n < \zeta$.
\item[2)] $h^1(\shE(n))\ne 0$ for $w_0\le n \le \bar\alpha-2$, and also for $n=\bar\alpha-1$ if $\zeta\notin\ZZ$.
\item[3)] If $\zeta\in\ZZ$ and $\alpha<\bar\alpha$, then $h^1(\shE(\bar\alpha-1))\ne 0$.
\end{description}
\end{theorem}
\begin{proof}
\textbf{1)} The Hilbert polynomial of the bundle $\shE$ is strictly negative for each integer such that  $w_0\le n < \zeta$, but for such  an integer $n$ we have $h^2(\shE(n))\ge0$ and $h^0(\shE(n))-h^0(\shE(-n+\epsilon-c_1)) \ge 0$ since $n\ge-n+\epsilon-c_1$ for every $n\ge w_0$, therefore we must have $h^1(\shE(n))\ne 0$.
\\
\textbf{2)} It is simply a restatement of 1) in term of $\bar\alpha$, which is, by definition, the integral part of $\zeta+1$.
\\
\textbf{3)} If $\zeta\in\ZZ$, then $\zeta=\bar\alpha-1$, so we have $\chi(\shE(\bar\alpha-1))=\chi(\shE(\zeta))=0$. Moreover $h^0(\shE(\bar\alpha-1))\ne0$ since $\alpha<\bar\alpha$, therefore $h^0(\shE(\bar\alpha-1))-h^3(\shE(\bar\alpha-1)) > 0$, and $h^1(\shE(n)) = 0$ implies $h^1(\shE(m)), \forall m \le n$; hence we must have $h^1(\shE(\bar\alpha-1))\ne 0$ to obtain  the vanishing of $\chi(\shE(\bar\alpha-1))$.
\end{proof}

\begin{corollary}
If $\shE$ is \emph{ACM}, then $\vartheta<0$.
\end{corollary}

\begin{remark}
Observe that in this section we assume $\alpha \ge \frac{\epsilon-c_1+5}{2}$, in order to have $w_0 \le \alpha-2$ and so to have a non-empty range for $n$ in Theorem \ref{stable1}.\end{remark} 

\begin{remark}
Observe that in the stable case we need not assume any vanishing of $h^1(\shO_X(n))$.
\end{remark}

\begin{remark}
Observe that split bundles are excluded in this section because they cannot be stable.
\end{remark}

\section{Examples}

We need the following 
 
\begin{remark}
Let $X \subset \PP^4$ be a smooth threefold of degree $d$ and let $f$ be the projection onto $\PP^3$  from a general point of $\PP^4$ not on $X$, and consider a normalized rank two vector bundle $\shE$ on $\PP^3$ which gives rise to the pull-back $\shF = f^*(\shE)$.
We want to check that $f_\ast (\mathcal {O}_X) \cong \oplus _{i=0}^{d-1} \mathcal {O}_{\mathbb {P}^3}(-i)$.
\\
Since $f$ is flat and $\deg (f)=d$, $f_\ast (\mathcal {O}_X)$ is a rank $d$
vector bundle. The projection formula and the cohomology of the
hypersurface
$X$ shows that $f_\ast (\mathcal {O}_X)$ is ACM. Thus there are integers $a_0\ge \cdots \ge a_{d-1}$
such that $f_\ast (\mathcal {O}_X) \cong \oplus _{i=0}^{d-1} \mathcal {O}_{\mathbb {P}^3}(a_i)$.
Since $h^0(X,\mathcal {O}_X)=1$, the projection formula gives $a_0 = 0$ and $a_i<0$ for all $i>0$.
Since $h^0(X,\mathcal {O}_X(1)) = 5 = h^0(\mathbb {P}^3,\mathcal {O}_{\mathbb {P}^3}(1))+h^0(\mathbb {P}^3,\mathcal
{O}_{\mathbb {P}^3})$, the projection formula gives $a_1=-1$ and $a_i \le -2$ for all $i \ge 2$. Fix
an integer $t \le d-2$ and assume proved $a_i = -i$ for all $i\le t$
and $a_i < -t$ for all $i>t$. Since $h^0(X,\mathcal {O}_X(t+1))= \binom{t+5}{4}
= 
\sum _{i=0}^{t} \binom{t+4-i}{3}$, we get $a_{t+1} = -t-1$ and, if $t+1 \le d-2$, $a_i<-t-1$ for all $i>t+1$.
Since $f_\ast (\mathcal {O}_X) \cong \oplus _{i=0}^{d-1} \mathcal {O}_{\mathbb {P}^3}(-i)$, the projection
formula gives the following formula for the first cohomology module: $$H^i(\shF(n)) = H^i(\shE(n)) \oplus H^i(\shE(n-1)) \oplus ...\oplus H^i(\shE(n-d+1))$$ 
all $i$.
Observe that, as a consequence of the above equalitiy for $ i = 0$, we obtain that  $\shF$ has the same $\alpha$ as $\shE$. Moreover the pull-back $\shF = f^*(\shE)$ and $\shE$ have the same Chern class $c_1$, while $c_2(\shF) = dc_2(\shE)$ and therefore $\delta(\shF) = d\delta(\shE)$.
\end{remark}

\noindent
\textbf{Examples}

\noindent
\textbf{1.} (a stable vector bundle with $c_1 = 0$, $c_2 = 4$ on a quadric hypersurface $X$).
\\
Choose $d = 2$ and take the pull-back $\shF$ of the stable vector bundle $\shE$ on $\PP^3$ of \cite{VV}, example 4.1. Then the numbers of  $\shF$ (see Notation)  are: $c_1 = 0$, $c_2 = 4$, $\alpha = 1$, $\bar\alpha = 2$,  $\zeta_0 = -\frac{3}{2}$, $w_0 = -1$, $\theta = \frac{25}{4}$, $\zeta = -\frac{3}{2}+\sqrt{\frac{25}{4}} = 1 \in \ZZ$. From \cite{VV}, example 4.1, we know that $h^1(\shE) \ne 0$. Since  $H^1(\shF(1)) = H^1(\shE(1)) \oplus H^1(\shE)$, we have:  $ h^1(\shF(1))\ne 0$, one shift higher than it is stated in Theorem \ref{stable2}, 2.
\\
\textbf{2.} (a non-stable vector bundle with $c_1 = 0$, $c_2 = 45$ on a hypersurface of degree $5$).
Choose $d = 5$ and take the pull-back $\shF$ of the stable vector bundle $\shE$ on $\PP^3$ of \cite{VV}, example 4.5. Then the numbers of  $\shF$ (see Notation)  are: $c_1 = 0$, $c_2 = 45$, $\alpha = -3$, $\delta = 90$, $\zeta_0 = 0$. From \cite{VV}, theorem 3.8, we know that $h^1(\shE(12)) \ne 0$. Since  $H^1(\shF(16)) = H^1(\shE(16)) \oplus \dots \oplus H^1(\shE(12))$, we have:  $ h^1(\shF(16))\ne 0$ (Theorem \ref{nonstable2} states that $h^1(\shF(10) \ne 0$).  
\\
\textbf{3.} (a stable vector bundle with $c_1 = -1$, $c_2 = 2$  on a quadric hypersurface).
\\
Let $\shE$ be the rank two vector bundle corresponding to the union of two skew lines on a smooth quadric hypersurface $Q \subset \PP^4$. Then its numbers are : $c_1 = -1$, $c_2 = 2$, $\alpha = 1$ and it is known that $h^1(\shE(n)) \ne 0$ if and only if $n = 0$.
\\
Observe that in this case $\theta = \frac{5}{2} \ge 0, \zeta_0 = -1$,  $\bar\alpha = 1$. Therefore theorem \ref{stable2} states exactly that $h^1(\shE(0)) \ne 0$, hence this example is sharp.
\\
\textbf{4.} (a non-stable vector bundle with $c_1 = 0$, $c_2 = 8$ on a quadric hypersurface).
\\
Choose $d = 2$ and take the pull-back $\shF$ of the non-stable vector bundle $\shE$ on $\PP^3$ of \cite{VV}, example 4.10. Then the numbers of  $\shF$ (see Notation)  are: $c_1 = 0$, $c_2 = 8$, $\alpha = 0$, $\zeta_0 = -\frac{3}{2}$, $\delta = 8$. We know (see \cite{VV}, example 4.10) that $h^1(\shE(2)) \ne 0, h^1(\shE(3)) = 0$.  Since  $H^1(\shF(3)) = H^1(\shE(3)) \oplus H^1(\shE(2))$, we have:  $ h^1(\shF(3)) \ne 0$, exactly the bound of Theorem \ref{nonstable2}.

\begin{remark}
The bounds for a degree $d$ threefold in $\PP^4$ agree with \cite{VV}, where $\PP^3$ is considered.
\end{remark}

\section{Threefolds with $\Pic(X) \ne \ZZ$}
 
Let $X$ be a smooth and connected projective threefold defined over an algebraically closed field
$\mathbb {K}$. Let $\Num(X)$ denote the quotient of $\Pic(X)$ by numerical equivalence.  Numerical classes are denoted by square brackets $[\,\,]$.
We assume $\Num(X) \cong \mathbb {Z}$ and take the unique isomorphism $\eta : \Num(X)\to \mathbb {Z}$ such that $1$ is the image of a fixed ample line bundle. Notice that $M\in \Pic(X)$ is ample if and only if $\eta ([M])>0$.

\begin{remark}\label{boundedness}
 Let $\eta : \Num(X) \to \mathbb {Z}$ be as before. Notice that every effective divisor on $X$ is ample and hence its $\eta$ is strictly positive. For any $t\in \mathbb {Z}$ set $\Pic_t(X):= \{L\in \Pic(X): \eta ([L])=t\}$. Hence
$\Pic_0(X)$ is the set of all isomorphism classes of numerically trivial line bundles on $X$.
The set $\Pic_0(X)$ is parametrized by a scheme of finite type (\cite{La}, Proposition 1.4.37).
Hence for each $t\in \mathbb {Z}$ the set $\Pic_t(X)$ is bounded.
Let now $\shE$ be a rank $2$ vector bundle on $X$. Since   $\Pic_1(X)$ is bounded there is a minimal integer $t$ such that  there is $B\in \Pic_t(X)$ and  $h^0(E\otimes B) >0$. Call it $\alpha (E)$ or just $\alpha$.
By the definition of $\alpha$ there is $B\in \Pic_{\alpha}(X)$ such that
$h^0(X,\shE\otimes B) >0$. Hence there is a non-zero map $j: B^\ast \to E$. Since
$B^\ast$ is a line bundle and $j\ne 0$, $j$ is injective. The definition of $\alpha$ gives
the non-existence of a non-zero effective divisor $D$ such that $j$ factors through an inclusion $B^\ast \to B^\ast (D)$, because
$\eta ([D]) >0$. Thus the inclusion $j$ induces
an exact sequence
\begin{equation}\label{eqa1}
0 \to B^\ast \to \shE \to \mathcal {I}_Z\otimes B \otimes \det (\shE) \to 0
\end{equation}
in which $Z$ is a closed subscheme of $X$ with pure codimension $2$.
\\
Observe that $\eta([B]) = \alpha, \eta([B^*]) = -\alpha, \eta([B \otimes det(\shE)]) =  \alpha +c_1$, hence the exact sequence is quite similar to the usual exact sequence that holds true in the case $\Pic(X) \cong \ZZ$.
\end{remark}

\noindent
NOTATION:
\\
We set $\epsilon:= \eta ([\omega _X])$, $\alpha := \alpha (\shE)$ and $c_1:= \eta ([\det (\shE)])$. So we can speak of a normalized vector bundle $\shE$, with $c_1 \in \{ 0,-1\}$. Moreover we say that $\shE$ is stable if $\alpha > 0$, nonstable if $\alpha \le 0$.
Moreover ª$\zeta_0, \zeta, w_0, \bar \alpha, \theta$ are defined as in section 2.  

\medskip
 
\begin{remark}
 Fix any $L \in \Pic_1(X)$ and set: $d = L^3 = $ degree of $X$.The degree $d$ does not depend on the numerical equivalence class. In fact, if $R$ is numerically equivalent to $0$, then $(L+R)^3 = L^3+R^3+3L^2R+3LR^2 = L^3+0+0+0 = L^3$. Then it is easy to see that the formulas for $\chi(\shO_X(n))$ and $\chi(\shE(n))$ given in section 2 still hold if we consider $\shO_X \otimes L^{\otimes n}$ and $\shE \otimes L^{\otimes n}$ (see \cite{Valenzano}). \end{remark}

\begin{remark}\label{a1}
\quad (a) Assume the existence of  $L\in \Pic(X)$ such that $\eta ([L]) =1$ and $h^0(X,L)>0$.
Then for every integer $t>\alpha $ there is $M\in \Pic(X)$ such that $\eta ([M])=t$
and $h^0(X,E\otimes M)>0$.

\quad (b) Assume $h^0(X,L)>0$ for every $L\in \Pic(X)$ such that $\eta ([L]) =1$. Then $h^0(X,E\otimes M)>0$
for every $M\in \Pic(X)$ such that $\eta ([M]) > \alpha$.
\end{remark}

\begin{proposition}\label{a2}
Let $\shE$ be a normalized rank two vector bundle and assume the existence of a spanned $R\in \Pic(X)$ such that $\eta ([R])=1$. If char $K > 0$, assume that $\vert R \vert$ induces an embedding of $X$ outside finitely many points. Assume 
\begin{equation}\label{eqb1}
2\alpha \le -\epsilon-3-c_1
\end{equation}
and $h^1(X,\shE\otimes N)=0$ for every $N\in \Pic(X)$ such that
$\eta ([N]) \in \{-\alpha -c_1-1,\alpha +2+e\}$.  If $h^1(X,B)=0$ for every $B\in \Pic(X)$ such that $\eta
([B])=-2\alpha -c_1$, then $\shE$ splits. 

If moreover $h^1(X,M)=0$ for every $M\in \Pic(X)$ then it is enough to assume that $h^1(X,\shE\otimes N)=0$ for every $N\in \Pic(X)$ such that
$\eta ([N]) = -\alpha -c_1-1$.
\end{proposition}

\begin{proof}
By assumption there is $M\in \Pic(X)$ such that $\eta ([M]) = \alpha$ and $h^0(X,\shE\otimes M)>0$. Set
$A:= M^\ast$. We have seen in remark \ref{boundedness} that $\shE$ fits into an extension of the following type:
\begin{equation}\label{b2}
0 \to A \to \shE \to \mathcal {I}_C\otimes \det (\shE)\otimes A^\ast \to 0
\end{equation}
with $C$ a locally complete intersection closed subscheme with pure dimension $1$.
Let
$H$ be a general element of $\vert R\vert$ and $T$ the intersection of $H$ with another
general element of $\vert R\vert$. Observe that $T$, under our assumptions, is generically reduced by Bertini's theorem-see \cite{HAL}, Theorem II, 8.18 and  Remark II, 8.18.1. Since $R$ is spanned, $T$ is a locally complete intersection
curve and $C\cap T=\emptyset$. Hence $\shE\vert T$ is an extension of $\det (\shE)\otimes A^\ast \vert T$ by $A\vert T$.  Since T is generically reduced and locally a complete intersection, it
is reduced. Hence $h^0(T,M^\ast ) = 0$ for every ample line bundle $M$ on $T$.
Since $\omega _T \cong (\omega _X\otimes R^{\otimes 2})\vert T$, we have $\dim (\mbox{Ext}^1(T,\det (\shE)\otimes A^\ast,A)
=h^0(T,\det (\shE)\otimes (A^\ast)^{\otimes 2} \otimes \omega _\otimes R^{\otimes 2}))\vert T)= 0$ (indeed $\eta ([\det (\shE)\otimes (A^\ast )^{\otimes 2}\otimes \omega _X\otimes R^{\otimes 2}])
= 2\alpha +c_1+e+2 <0$). Hence $\shE\vert T \cong A\vert T\oplus (\det (\shE)\otimes A^\ast )\vert T$. Let $\sigma$ be the non-zero
section of $(\shE\otimes (A\otimes \det (\shE)^\ast )\vert T$ coming from the projection onto the second factor of the decomposition
just given. The vector bundle $\shE\vert H$
is an extension of $\det (\shE)\otimes A^\ast \vert H$ by $A\vert H$ if and only if $C\cap H =\emptyset$. Since
$R$ is ample, $C\cap H = \emptyset$ if and only if $C=\emptyset$. Hence
we get simultaneously $C\cap H=\emptyset$ and $\shE\vert H \cong A\vert H\oplus \det (\shE)\otimes A^\ast \vert H$
if we prove the existence of $\tau \in H^0(H,(\shE\otimes (A\otimes \det (\shE)^\ast )\vert H)$ such that $\tau \vert T = \sigma$.
To get $\tau$ it is sufficient to have $H^1(H,(E\otimes (A\otimes \det (\shE)^\ast \otimes R^\ast )\vert H) =0$.
A standard exact sequence shows that $H^1(H,(\shE\otimes (A\otimes \det (\shE)^\ast \otimes R^\ast )\vert H) =0$
if $h^1(X,(\shE\otimes (A\otimes \det (\shE)^\ast \otimes R^\ast ) =0$ and $h^2(X,(\shE\otimes (A\otimes \det (\shE)^\ast \otimes R^\ast
\otimes R^\ast) =0$. Since $\shE^\ast \cong \shE\otimes \det (\shE)^\ast$, Serre duality
gives $h^2(X,(E\otimes (A\otimes \det (\shE)^\ast \otimes R^\ast
\otimes R^\ast) =h^1(X,\shE\otimes A\otimes R^{\otimes 2}\otimes \omega _X)$. Since $\eta ([A\otimes \det (\shE)^\ast \otimes R^\ast ])=-\alpha -c_1-1$
and $\eta ([A\otimes R^{\otimes 2}\otimes \omega _X] )= \alpha +e+2$, we get that $C=\emptyset$. The last sentence follows because
$\eta ([A^{\otimes 2}\otimes \det (\shE)^\ast] )= -2\alpha -c_1$.
\end{proof}

\begin{remark}\label{a3}
Instead of the smoothness of $X$ we may assume that $X$ is locally algebraic factorial, i.e. that all local rings
$\mathcal {O}_{X,P}$ are factorial. This assumption seems to be essential, because without it a non zero section
of $E\otimes M$ with $\eta ([M]) = \alpha (E)$ could vanish on an effective Weil divisor and hence
we could not claim the existence of the exact sequence (\ref{b2}). 
\end{remark}
 
 \begin{remark}\label{a4}
Fix integers $t < z \le \alpha -2$. Assume the existence of $L \in \Pic(X)$ such that
$\eta ([L])=z$ and $h^1(X,E\otimes L)=0$. If there is $R\in \Pic(X)$ such that $\eta ([R])=1$
and
$h^0(X,R)>0$, then there exists $M\in \Pic(X)$ such that $\eta ([M]) =t$ and $h^1(X,E\otimes M)=0$.
If $h^0(X,R)>0$ for every $R\in \Pic(X)$ such that $\eta ([R])=1$, then $h^1(X,E\otimes M)=0$ for every $M\in
\Pic(X)$ such that $\eta ([M]) =t$.

The proof can follow the lines of Lemma \ref{leftvanishing}. In fact  consider a line bundle  $R$ with $\eta([R]) = 1$ and let $H$ be the zero-locus of a non-zero section of $R$; then we have the following exact sequence:
$$0 \to \shE \otimes L \to \shE \otimes L \otimes R \to (\shE \otimes L\otimes R)_H \to 0.$$ 
Now observe that the vanishing of $h^1(X,\shE \otimes L)$ implies that $h^0(\shE \otimes L \otimes R)_H = 0$. And now we can argue as in Lemma \ref{leftvanishing} (see also \cite{VV}).
\end{remark}

\begin{remark}\label{a6}
\quad (a) Assume the existence of  $L\in \Pic(X)$ such that $\eta ([L]) =1$ and $h^0(X,L)>0$.
Then for every integer $t>\alpha $ there is $M\in \Pic(X)$ such that $\eta ([M])=t$
and $h^0(X,E\otimes M)>0$.

\quad (b) Assume $h^0(X,L)>0$ for every $L\in \Pic(X)$ such that $\eta ([L]) =1$. Then $h^0(X,E\otimes M)>0$
for every $M\in \Pic(X)$ such that $\eta ([M]) > \alpha$.
\end{remark}

\begin{remark}In all our results of sections 4 and 5 we use the vanishing of $h^1(\shO_X(n))$ (and by Serre duality of $h^2(\shO_X(n))$), $\forall n$ (or, at least, $\forall n\notin \{0,\cdots,\epsilon\}$), see Remark 4.12.
 
From now on we need to use similar vanishing conditions and so we introduce the following condition:
 
\quad $\mathbf{(C4)}$ $h^1(X,L) =0$ for all $\Pic(X)$ such that either $\eta ([L]) <0$
or $\eta ([L]) >\epsilon$.
\\
 Observe that $\mathbf{(C4)}$ is always satisfied in characteristic 0 (by the Kodaira vanishing theorem). In positive characteristic it is often satisfied. This is always the case if $X$ is an abelian variety (\cite{Mumford} p. 150).
 \\
 Observe also that, if $\epsilon \le -1$, the Kodaira vanishing and our condition put no restriction on $n$ (see also Remark 4.12).
\end{remark}

\noindent
\textbf{Example}.  If (\ref{eqb1}) holds, then $-2\alpha -c_1 > \epsilon$. Hence we may apply Proposition \ref{a2} to $X$.
In particular observe that, in the case of an abelian variety with $\Num(X) \cong \mathbb {Z}$ or in the case
of a Calabi-Yau threefold with $\Num(X) \cong \mathbb {Z}$,  we have $\epsilon = 0$. 
Notice that Proposition \ref{a2} also applies to any threefold $X$ whose $\omega _X$ is a torsion sheaf.

\medskip
 
   With the assumption of condition $\mathbf{(C4)}$ the proofs of Theorems \ref{nonstable1}, \ref{nonstable2}, \ref{nonstable3} can be easily modified in order to obtain the  statements below ($\shE$ is normalized, i.e. $\eta ([\det (\shE )]) \in \{-1,0\}$), where, by the sake of simplicity, we assume $\epsilon \ge 0$ (if $\epsilon < 0$, $\mathbf{(C4)}$, which holds by \cite{SB}, implies that all the vanishing of $h^1$ and $h^2$ for all $L\in \Pic(X)$ hold).

\begin{theorem}\label{v1}
Assume $\mathbf{(C4)}$, $\alpha \le 0$, the existence of $R\in \Pic(X)$ such that $\eta ([R])=1$ and $\zeta _0 < -\alpha -c_1-1$. Fix an integer
$n$ such that $\zeta _0 < n \le -\alpha - 1 -c_1$. Fix
$L\in \Pic(X)$ such that $\eta ([L]) = n$. Then $h^1(\shE \otimes L) \ge (n-\zeta _0)\delta >0$. 
\end{theorem}

\begin{remark} Observe that we should require the following conditions: $n-\alpha \notin \{0,\dots ,\epsilon \}, \epsilon-n+\alpha  \notin \{0,\dots ,\epsilon \}$. But they are automatically fulfiled under the  assumption that $\zeta_0 < -\alpha-c_1-1$.
\end{remark} 

\begin{theorem}\label{v2}
Assume $\mathbf{(C4)}$, $\alpha \le 0$, the existence of $R\in \Pic(X)$ such that $\eta ([R])=1$ and the same hypotheses of Theorem \ref{nonstable2}. Fix
$L\in \Pic(X)$ such that $\eta ([L]) = n$. Then $h^1(\shE \otimes L) \ge
-S(n) >0$ ($S(n)$ being defined as in Theorem \ref{nonstable2}).
\end{theorem}

\begin{theorem}\label{v3}
Assumption as in Theorem \ref{nonstable3}. Moreover assume $\mathbf{(C4)}$ and
 $n -\alpha \notin \{0,\dots ,\epsilon \}$. Fix
$L\in \Pic(X)$ such that $\eta ([L]) = n$. Then $h^1(\shE \otimes L) \ge
-\frac{d}{6}F(n+\alpha -\zeta _0)>0$ ($F$ being defined as in Theorem \ref{nonstable3}).  
\end{theorem}

\begin{remark} Observe that in Theorems \ref{v2} and \ref{v3} we should require  $n -\alpha \notin \{0,\dots ,\epsilon \}$, but the assumption $\epsilon-\alpha-c_1+1 \le n$ implies that it is automatically fulfilled.  Observe that in Theorems \ref{v2} and \ref{v3} we require  $n -\alpha \notin \{0,\dots ,\epsilon \}$, but the assumption $\epsilon-\alpha-c_1+1 \le n$ implies that the requirement is automatically fulfilled.

\end{remark}

The proofs of the above theorems are based on the existence of the exact sequence (\ref{eqa1}) and on the properties of $\alpha$. They follow the lines of the proofs given in the case $\Pic(X) \cong \ZZ$. Here and in section 4 we actually need only the Kodaira vanishing (true in characteristic 0 and assumed in characteristic $p > 0$) and no further vanishing of the first cohomology.
 
Also the stable case can be extended to a smooth threefold with $\Num(X) \cong \ZZ$. Observe that the proofs can follow the lines of the proofs given in the case  $\Pic(X) \cong \ZZ$ and make use of Remark 7.6 (which extends \ref{leftvanishing}).

More precisely we have:

\begin{theorem}\label{v4}
Assumptions as in \ref{stable1} and fix $L\in \Pic(X)$ such that $\eta ([L]) = n$. Then, if $\alpha\ge\frac{\epsilon+5-c_1}{2}$, then $h^1(\shE\otimes L)\ne 0$ for $w_0\le n\le \alpha-2$.
\end{theorem}

\begin{theorem}\label{v4} Assumptions as in \ref{stable2} and fix $L\in \Pic(X)$ such that $\eta ([L]) = n$. Then  the following hold:
\begin{description}
\item[1)] $h^1(\shE\otimes L)\ne 0$ for $\zeta_0< n < \zeta$.
\item[2)] $h^1(\shE \otimes L)\ne 0$ for $w_0\le n \le \bar\alpha-2$, and also for $n=\bar\alpha-1$ if $\zeta\notin\ZZ$.
\item[3)] If $\zeta\in\ZZ$ and $\alpha<\bar\alpha$, then $h^1(\shE \otimes N) \ne 0$, for every $N$ such that $\eta([N]) = \bar\alpha-1$.
\end{description}
\end{theorem}

\begin{remark}\label{a5} 
The above theorems can be applied to any $X$ such that  $\Num(X)$ $\cong \mathbb {Z}$, $\epsilon =0$ and $h^1(X,L)=0$
for all $L\in \Pic(X)$ such that $\eta ([L]) \ne 0$, for instance to 
 $X = $ an abelian threefold with $\Num(X) \cong \mathbb {Z}$.
 \end{remark}
 
 \begin{remark}
 If $X$ is any threefold (in characteristic $0$ or positive) such that $h^1(X,L) = 0, \forall L \in \Pic(X)$, then we can avoid the restriction $n-\alpha \notin \{0,...,\epsilon\}$. Not many threefolds, beside any $X \subset \PP^4$, fulfil these conditions.
 \end{remark}
 
 \begin{remark} 
 
 Observe that in Theorem \ref{v4} we do not assume $\mathbf{(C4)}$ (see also remark 5.8) \end{remark}
 
 \begin{remark}Observe that also in the present case ($\Num(X) \cong \ZZ$), we have: $\delta = 0$ if and only if $\shE$ splits. Therefore Remarks 4.13 and 5.9  apply  here. 
 \end{remark}

\bigskip\bigskip

\noindent
\textsc{Edoardo BALLICO} \\
Dipartimento di Matematica, Universit\`a di Trento \\
via Sommarive 14, 38050 Povo (TN), Italy \\
e-mail: \texttt{ballico@science.unitn.it}
\\[10pt]
\textsc{Paolo VALABREGA} \\
Dipartimento di Matematica, Politecnico di Torino \\
Corso Duca degli Abruzzi 24, 10129 Torino, Italy \\
e-mail: \texttt{paolo.valabrega@polito.it}
\\[10pt]
\textsc{Mario VALENZANO} \\
Dipartimento di Matematica, Universit\`a di Torino \\
via Carlo Alberto 10, 10123, Torino, Italy \\
e-mail: \texttt{mario.valenzano@unito.it}

\end{document}